\documentclass[a4paper]{amsart}
\usepackage[latin9]{inputenc}
\usepackage{amsfonts}
\usepackage{amssymb}
\usepackage{amsmath}
\usepackage{amsthm}
\usepackage{geometry}
\usepackage{verbatim}
\usepackage[toc,page]{appendix}
\begin{document}
\providecommand{\keywords}[1]{\textbf{\textit{Keywords: }} #1}
\newtheorem{theorem}{Theorem}[section]
\newtheorem{lemma}[theorem]{Lemma}
\newtheorem{prop}[theorem]{Proposition}
\newtheorem{corollary}[theorem]{Corollary}
\theoremstyle{definition}
\newtheorem{defi}[theorem]{Definition}
\theoremstyle{remark}
\newtheorem{remark}[theorem]{Remark}
\newtheorem{example}[theorem]{Example}
\newtheorem{problem}[theorem]{Problem}
\newtheorem{question}[theorem]{Question}
\newtheorem{conjecture}[theorem]{Conjecture}
\newtheorem{condenum}[theorem]{Condition}

\newcommand{\cc}{{\mathbb{C}}}   
\newcommand{\ff}{{\mathbb{F}}}  
\newcommand{\nn}{{\mathbb{N}}}   
\newcommand{\qq}{{\mathbb{Q}}}  
\newcommand{\rr}{{\mathbb{R}}}   
\newcommand{\zz}{{\mathbb{Z}}}  
\newcommand{\fp}{{\mathfrak{p}}}

\title{On the mod-$p$ distribution of discriminants of $G$-extensions}
\author{Joachim K\"onig}
\email{jkoenig@knue.ac.kr}
\address{Korea National University of Education, Department of Mathematics Education, 28173 Cheongju, South Korea}
\begin{abstract}
This paper was motivated by a recent paper by Krumm and Pollack (\cite{KP18}) investigating modulo-$p$ behaviour of quadratic twists with rational points of a given hyperelliptic curve, conditional on the abc-conjecture. We extend those results to twisted Galois covers with arbitrary Galois groups. The main point of this generalization is to interpret those results as statements about the sets of specializations of a given Galois cover under restrictions on the discriminant. In particular, we make a connection with existing heuristics about the distribution of discriminants of Galois extensions such as the Malle conjecture: our results show in a precise sense the non-existence of ``local obstructions" to such heuristics, in many cases essentially only under the assumption that $G$ occurs as the Galois group of a Galois cover defined over $\qq$. This complements and generalizes a similar result in the direction of the Malle conjecture by D\`ebes (\cite{Debes17}).\end{abstract}
\maketitle

\section{Introduction}
\subsection{Distribution of Galois extensions}
Let $G\le S_n$ be a finite transitive permutation group, always assumed non-trivial in the following. By a $G$-extension of $\qq$, we mean a degree-$n$ extension whose Galois closure has group permutation-isomorphic to $G$. The distribution of discriminant values of $G$-extensions of $\qq$ is one of the great mysteries of inverse Galois theory.
In \cite{Malle02}, Malle presented a precise conjecture on the asymptotic number of $G$-extensions with a bounded discriminant:
\begin{conjecture}[Malle]
\label{conj:malle}
Fix a finite group $G$ and $\epsilon>0$. Let $\mathcal{N}(G,B)$ be the number of Galois extensions of $\mathbb{Q}$ with discriminant of absolute value $\le B$. Then there exist constants $c_1(G)$ and $c_2(G,\epsilon)$ such that 
$$c_1(G) B^{\alpha(G)} \le \mathcal{N}(G,B) \le c_2(G,\epsilon) B^{\alpha(G)+\epsilon},$$
for all sufficiently large $B$.\\
Here $\alpha(G):= (\min_{1\ne g\in G} {\textrm{ind}}(g))^{-1}$, where the index ${\textrm{ind}}(g)$ of a permutation $g$ is defined as the minimal number of transpositions in a product representing $g$.
\end{conjecture}

The conjecture in this form is known in full for nilpotent groups in regular action (\cite{KlM}), and for some very small degree groups (e.g., \cite{Bhargava}, \cite{Bhargava2}), but wide-open in general.

We note here that the exponent $\alpha(G)$ is due to the fact that (tamely) ramified primes are weighted differently by the discriminant, depending on the cycle structure of their inertia group generator. This weighting is not always the ideal one for counting purposes, and in fact more general weightings were proposed e.g.\ in \cite{EV}. In the following, we will not distinguish between different inertia groups, i.e., use a ``trivial" weighting. This amounts to counting $G$-extensions $F/\qq$ by their {\textit{reduced discriminant}} $\delta(F/\qq)$, which we define as the product of all ramified prime numbers without exponents.\footnote{There are also slightly different definitions of reduced discriminants. We use the above notion due to our motivation, but note that our proof in Section \ref{sec:proof} requires only that a prime $p$ strictly divides (resp., does not divide) $\delta$ if $p$ is tamely ramified (resp., unramified) and that the exact exponent of $p$ in $\delta$ is determined by the completion of $F/\qq$ at $p$.} Note that the reduced discriminant does not change upon taking Galois closures, whence in the following we will consider Galois extensions (i.e., $G$ in its regular action) unless stated otherwise.

The analog of (the lower bound in) Malle's conjecture for this situation (in fact, already implied in Malle's counting heuristics) then states that the number of $G$-extensions of reduced discriminant at most $B$ should be at least asymptotic to $cB$ for some positive constant $c$.

This expectation is essentially based on two general heuristic assumptions, suggested by Malle (see \cite{Malle04}) and later extended by, e.g., work of Bhargava:
a) the assumption that all squarefree positive integers (apart from trivially obstructed ones) should have a ``positive probability" to occur as reduced discriminants of $G$-extensions, and b) the additional assumption that any cyclic subgroup of $G$ should occur as inertia group at a prime $p$ with positive probability (as long as there is no trivial local obstruction at $p$), thereby accordingly increasing the expected total number of $G$-extensions of a fixed discriminant value. To make these intuitive statements concrete, consider as a toy example the group $G=C_3$. Here, only primes $\equiv 1$ mod $3$ (and $3$ itself) can ramify in $G$-extensions, so squarefree integers divisible by any prime $\equiv 2$ mod $3$ are ``trivially obstructed". The group-theoretical reason is that the non-trivial conjugacy classes of $C_3$ are not $\qq$-rational. On the other hand, this is really the only obstruction in this case, since all squares of squarefree integers all of whose prime factors are congruent to $1$ modulo $3$ do occur as discriminants of $C_3$-extensions.
\footnote{Indeed, if $n=p_1\cdots p_r$ with $p_i\equiv 1$ mod $3$ for all $i=1,\dots, r$, then each cyclotomic extension $\qq(\zeta_{p_i})/\qq$ contains a $C_3$-subextension ramified only at $p_i$. Their compositum has group $C_3^r$, and the fixed field of the augmentation ideal is a $C_3$-extension of $\qq$ ramified exactly at all the $p_i$.}
 Integers with these restrictions on prime factors would still form a density zero subset of all integers; however, integers with more than one prime factor occur several times (according to the different possibilities of inertia conjugacy classes), bringing the total count $\mathcal{N}(C_3,B)$ to $c\cdot B$ asymptotically.

More generally, given a group $G$ and a non-trivial conjugacy class $C$ of elements of order $e$, the class $C$ has a field of rationality $K(C)\subseteq \qq(\zeta_e)$, meaning that only primes which split in $K(C)/\qq$ can be tamely ramified with inertia group generator in $C$ in any $G$-extension. We then denote by $\mathcal{P}(G)$ the set of all primes which split completely in at least one field $K(C)$ for a non-identity conjugacy class of $G$.

While the assumption b) above is responsible for the concrete exponent in Malle's conjecture, part a) is interesting in its own right. 
Below, we present a) formally as a conjecture in general form.

\begin{conjecture}
\label{conj:malle2}
Let $G$ be a finite group, let $\mathbb{N}_G$ be the set of all squarefree integers all of whose prime factors either divide $|G|$ or lie in $\mathcal{P}(G)$, and let $S(G)$ be the set of all positive integers which occur as the reduced discriminant of a $G$-extension. Then $S(G)$ has positive density inside $\mathbb{N}_G$.
\end{conjecture}

As seen above, the somewhat delicate wording of Conjecture \ref{conj:malle2} is due to the presence of certain {\textit{local obstructions}} for values of reduced discriminants. Formally, given a prime $p$ 
and a residue class $k$ mod $p$, we call $k$ mod $p$ a local obstruction for reduced discriminants of $G$-extensions, if no $G$-extension $F/\qq$ has reduced discriminant congruent to $k$ mod $p$. If $p\notin \mathcal{P}(G)$ and $p\not{|} |G|$, then $0$ mod $p$ is a local obstruction as explained above; we call this case a {\textit{trivial local obstruction}}. In order to reach the density predicted by Conjecture \ref{conj:malle2}, it is necessary (although of course not sufficient) that there are not too many local obstructions. Indeed, the average number of non-trivial local obstructions at a prime $p\le N$ needs to converge to $0$ for $N\to\infty$, since every local obstruction at $p$ decreases the candidate set of reduced $G$-discriminants $\le N$ (with $N$ sufficiently large) by a factor $(1-1/p)$, and $\prod_{p\in S(N)\subset\{1,...,N\}}(1-1/p)\to 0$ for any positive density subsets $S(N)$ of $\{1,...,N\}$.

In fact, it seems reasonable to conjecture:
\begin{conjecture}[``Local Pre-Malle"]
\label{conj:localmalle}
Let $G$ be a finite group. Given any prime $p$, not in some finite set $S_0$ depending on $G$, and an integer $k \not\equiv 0$ mod $p$,
there exists a $G$-extension of $\qq$ whose reduced discriminant $\delta$ fulfills $\delta \equiv k$ mod $p$. Furthermore, if $p\in \mathcal{P}(G)$, then the same holds for the zero residue class. \\
In other words, there are only finitely many local obstructions for reduced discriminants of $G$-extensions, apart from the trivial ones.
\end{conjecture}

\subsection{Outline of the paper}
The main contribution of this paper is the observation that Conjecture \ref{conj:localmalle} holds (conditionally on the abc-conjecture) for groups $G$ which occur as regular Galois groups over $\qq$ (Theorem \ref{thm:main} and Corollary \ref{cor1}), and even unconditionally under some further light assumptions (Lemma \ref{lem:degle3}). In other words, for such groups there are at least no local reasons why the Malle-type conjectures on distribution of discriminants of $G$-extensions (as in Conjecture \ref{conj:malle2}) should fail. In Section \ref{sec:twists}, we give some background on the connection of specializations and twists of Galois covers, thereby re-interpreting our main results as existence results for rational points on twisted covers, directly generalizing the main results in \cite{KP18}.

We also present several variants and strengthenings of our main results, e.g.\ a combination with a previous counting result on specializations by D\`ebes (Theorem \ref{thm:main_strong2}) and some considerations on the set of exceptional primes in Section \ref{sec:exc}. Since the proofs of these variants share the same fundamental ideas, these are presented step-by-step in Section \ref{sec:proof}, and then later adjusted appropriately for the specific variants.

We also refer to Conjecture \ref{conj:localmalle2} in Section \ref{sec:strong} for an analog in terms of mod-$p$ distribution of the usual, non-reduced, discriminant.

\section{Main result}
One of our main goals is to essentially reduce Conjecture \ref{conj:localmalle} to the regular inverse Galois problem for $G$, i.e., to the question whether $G$ occurs as the Galois group of a $\qq$-regular Galois extension of $\qq(T)$. By a $\qq$-regular $G$-extension, we mean a Galois extension $E/\qq(T)$ with Galois group $G$ and with $E\cap \overline{\qq} = \qq$. If such an extension exists, we say that $G$ is a regular Galois group over $\mathbb{Q}$.

\begin{theorem}
\label{thm:main}
Let $E/\qq(T)$ be a $\qq$-regular $G$-extension. 
Then the following hold:
\begin{itemize}
\item[i)] Assume that the abc-conjecture holds. Then there exists $N_0\in \nn$ such that for every prime $p\ge N_0$ and every $k\not\equiv 0$ mod $p$,
there are infinitely many specializations $E_{t_0}/\qq$ of $E/\qq(T)$ which have Galois group $G$ and whose reduced discriminant is congruent to $k$ mod $p$.\\
More precisely, the set of all specializations $E_{t_0}/\qq$ with the above properties and with discriminant of absolute value $\le B$ is asymptotically of size at least $B^\alpha$, for some positive constant $\alpha$ depending only on $E/\qq(T)$ (and not on $p$!).
\item[ii)] If in addition $E/\qq(T)$ has at least one $K(C)$-rational branch point for each conjugacy class $C$ of $G$ \footnote{Note that we do not require this branch point to actually have inertia group generator in class $C$. In particular, the condition is automatically fulfilled if $E/\qq(T)$ has a $\qq$-rational branch point; however, the latter is not possible for certain groups by the so-called branch cycle lemma, hence our more general condition.}  
and if $N_0\le p\in \mathcal{P}(G)$, then the same conclusion holds (unconditionally) for $k=0$.
\end{itemize}
\end{theorem}

The proof of Theorem \ref{thm:main} is contained in Section \ref{sec:proof}.

\begin{corollary}
\label{cor1}
Assume $G$ occurs as the Galois group of a $\qq$-regular Galois extension $E/\qq(T)$, and assume that $E/\qq(T)$ has a $K(C)$-rational branch point for each conjugacy class $C$ of $G$. 
Then Conjecture \ref{conj:localmalle} holds for $G$, conditional on the abc-conjecture. Furthermore, it suffices to consider $G$-extensions of $\qq$ which arise as specializations of $E/\qq(T)$.
\end{corollary}

\begin{remark}
\begin{itemize}
\item[a)]
Of course one point of Theorem \ref{thm:main} is to give a distribution result in the direction of (but of course much weaker than) the Malle conjecture. In this regard, it cannot yield new insights for groups for which these heuristics are known in full. E.g., for abelian groups Malle's conjecture is known to hold due to Wright (\cite{Wright}), and indeed here Conjecture \ref{conj:localmalle} is trivial: namely, one easily reduces to finding cyclic extensions of order $n$ with reduced discriminant contained in a prescribed coprime residue class modulo $p$. Now as soon as $p$ is coprime to $n$, this is achieved via Dirichlet's prime number theorem, since for any $q\equiv 1$ mod $n$, the cyclotomic extension $\qq(\zeta_q)/\qq$ contains a $C_n$-subextension ramified only at $q$.\\
However, Theorem \ref{thm:main} should also be viewed under the general question: how much variety is contained in the set of all specializations of a given regular Galois extension?
Already Hilbert's irreducibility theorem gave a partial answer to this question, ensuring the existence of infinitely many specializations with the same Galois group. Recent work by the author and others investigates the question whether the specialization set of a single regular Galois extension can answer (the lower bound of) Malle's conjecture in full (\cite{KL19}), or even contain all extensions of a prescribed Galois group (\cite{KL18}, \cite{KLN18}), giving negative answers under weak assumptions in both cases. The above theorem complements these negative {\textit{global}} results with a positive result in the direction of the {\textit{local}} part of Malle's conjecture.
\item[b)] The special case $G=C_2$ in Theorem \ref{thm:main}i) yields Krumm's and Pollack's result about rational points on quadratic twists on hyperelliptic curves (\cite[Theorem 2]{KP18}), as explained in Section \ref{sec:twists}. Compare also the earlier \cite{Krumm} for some even more special cases.
\end{itemize}
\end{remark}
 
 \section{Preliminaries on twists and specializations}
 \label{sec:twists}
We explain briefly in what sense our results generalize previous results about quadratic twists of hyperelliptic curves.
 Compare \cite[Section 2]{Debes17} for the following.\\
Let $f:X\to \mathbb{P}^1$ be a Galois cover defined over $\qq$ with Galois group $G$ (for short: a $\qq$-$G$-cover) and with branch points $t_1,\dots,t_r\in \mathbb{P}^1(\overline{\qq})$. 
Given any $t_0 \in \mathbb{P}^1(\qq) \setminus \{t_1, \dots, t_r\}$, there is a homomorphism $f_{t_0}: G_\qq \to G$ (well-defined up to inner conjugation in $G$), called the {\textit{specialization representation}} of $f$ at $t_0$. In the same way, if $E/\qq(T)$ denotes the function field extension of the cover $f$,
there is a {\textit{specialization}} of $E/\qq(T)$ at $t_0$, defined as the residue field extension of $E/\mathbb{Q}(T)$ at a prime ideal $\mathcal{P}$ lying over $\langle T-t_0 \rangle$, and denoted by $E_{t_0}/\qq$. As the extension $E/\qq(T)$ is Galois, the field $E_{t_0}$ does not depend on the choice of $\mathcal{P}$ and the extension $E_{t_0}/\qq$ is finite and Galois. Moreover,  its Galois group is a subgroup of ${\textrm{Gal}}(E/\qq(T))$, namely the decomposition group of $E/\qq(T)$ at $\mathcal{P}$.
The two notions of specializations are compatible in the sense that the fixed field of the kernel of $f_{t_0}$ equals $E_{t_0}$. 

Given a $\qq$-$G$-cover $f:X\to\mathbb{P}^1$ and a homomorphism $\varphi:G_\qq\to G$, there exists a (not necessarily Galois) cover $\tilde{f}^\varphi:\tilde{X}\to \mathbb{P}^1$ over $\qq$, called the {\textit{twisted cover}} (or just, the twist) of $f$ by $\varphi$ with the following properties:
\begin{itemize}
\item[a)] $\tilde{f}^\varphi$ becomes isomorphic to $f$ after base change to the fixed field of the kernel of $\varphi$.
\item[b)] For any non-branch point $t_0\in \mathbb{P}^1(\qq)$, the specialization representation $f_{t_0}$ equals $\varphi$ if and only if $\tilde{f}^\varphi$ has a rational point extending $t_0$.
\end{itemize}

The above translations allow to directly rephrase our main results in terms of distribution of twists with rational points of a given Galois cover; e.g., Theorem \ref{thm:main}i) becomes:

{\textbf{Theorem \ref{thm:main}'}}:\\
Let $f$ be a $\qq$-$G$-cover. Assuming the abc-conjecture, there exists $N_0\in \nn$ such that for every prime $p\ge N_0$ and every $k\not\equiv 0$ mod $p$,
there are infinitely many epimorphisms $\varphi: G_\qq \to G$ such that
\begin{itemize}
\item[a)] (the fixed field of the kernel of) $\varphi$ has reduced discriminant congruent to $k$ modulo $p$, 
\item[b)] the twisted cover $\tilde{f}^\varphi$ has a non-trivial rational point (i.e., a rational point not extending a branch point).
\end{itemize}

In particular, in the case $G=C_2$, $\qq$-$G$-covers $f:X\to \mathbb{P}^1$ correspond one-to-one to hyperelliptic curve equations $\mathcal{C}: X^2 = F(T)$, with $F\in \zz[T]$ squarefree, and homomorphisms $\varphi:G_\qq\to G$ correspond one-to-one to squarefree integers $d\in \zz$. The twisted cover of $f$ by $\varphi$ is then nothing but the quadratic twist $\mathcal{C}_d: X^2 = dF(T)$. Asking whether the twist $\mathcal{C}_d$ has a non-trivial (i.e., unramified) rational point is therefore the same as asking whether the cover $f$ specializes to the quadratic extension $\qq(\sqrt{d})/\qq$. This extension, in term, has reduced discriminant ``essentially" equal to $d$ (possibly up to sign and/or a factor $2$; see however footnote 1, asserting that instead defining the reduced discriminant of $\qq(\sqrt{d})$ to simply be $d$ would not affect the proof in Section \ref{sec:proof}). The special case $G=C_2$ of Theorem \ref{thm:main}' therefore becomes:

\begin{theorem}[\cite{KP18}, Theorem 2]
\label{thm:kp}
Let $f\in \mathbb{Z}[X]$ be non-constant and separable. Conditional on the abc-conjecture, for every sufficiently large prime $p$ and every $k\not\equiv 0$ mod $p$, there exist infinitely many squarefree $d\in \mathbb{Z}$ such that $d\equiv k$ mod $p$ and the hyperelliptic curve $Y^2=df(X)$ has a non-trivial rational point.\footnote{In fact, \cite[Theorem 2]{KP18} reaches this conclusion about {\it integral} points, a strengthening of interest in general, although not as relevant in our Galois-theoretical setup.}
\end{theorem}

 \section{Proof of Theorem \ref{thm:main}}
 \label{sec:proof}
 We note that the novelty of Theorem \ref{thm:main} lies in assertion i). Assertion ii) is essentially known and included in the statement only to complete the picture,
 
 Several main ideas of the proof below are present in some form in the proof of the special case $G=C_2$ in \cite{KP18}. The generalization requires some care though, e.g.\ regarding the existence of ``bad primes" for regular Galois extensions. The eventual goal is to reduce the problem to squarefree values of a certain two-variable polynomial. Existence of such specializations is known conditional on the abc-conjecture (first noted in the one-variable case by Granville in \cite{Granv}), of course under the trivially necessary condition that the polynomial has no fixed divisor $p^2$ (i.e., no prime value whose square divides all specialization values). 
 Here we call an integer $d>1$ a {\textit{fixed divisor}} of a polynomial $f(X_1,\dots, X_r)$ if $d$ divides all values $f(x_1,\dots,x_r)$ with $x_1,\dots, x_r\in \zz$.
 Note that the set of fixed divisors of $f$ is always finite, which is obvious from just considering a single non-zero specialization.

To prove Theorem \ref{thm:main}i), let $t_1,...,t_s \in \mathbb{P}^1(\overline{\qq})$ be a set of $G_\qq$-orbit representatives of the branch points of $E/\qq(T)$, and let $f_i \in \zz[X,Y]$ be the homogenized irreducible polynomial of $t_i$ over $\zz$ (in particular, $f_i:=Y$ for $t_i = \infty$). Let $F(X,Y) = \prod_{i=1}^s f_i$, and note that $F$ has no constant integer factor, i.e., is of content $1$. Then the following is well-known (compare e.g. \cite[Prop.\ 4.2]{Beckmann}, or \cite[\S 2.2.3]{Legrand} for stronger statements):
 
{\textbf{Specialization inertia theorem}}:\\
There exists a finite set $S_0$ of primes such that for all $p\notin S_0$ and for all $t_0 = (x_0:y_0)\in \mathbb{P}^1(\qq)$ which are not branch points of $E/\mathbb{Q}(T)$ (with $(x_0,y_0)=1$), the following holds:
 If $F(x_0,y_0)$ is strictly divisible by $p$, then $p$ is ramified in $E_{t_0}/\qq$; and if $F(x_0,y_0)$ is not divisible by $p$, then $E_{t_0}/\qq$ is unramified at $p$.

 In order to control ramification in $E_{t_0}/\qq$, we aim for values $(x_0:y_0)$ which make $F(x_0,y_0)$ squarefree. 
 
 {\textit{First step: Eliminating fixed prime divisors}}\\
 At first, we ensure that $F$ itself has no fixed prime divisors (i.e., in particular no primes whose square divides all values $F(x_0,y_0)$). This is of course not necessarily the case for the given $F$, 
 but can be guaranteed via a linear transformation in the variable $T$ (and therefore may be assumed without loss), as follows:
  \begin{lemma}
 \label{lem:pol}
 Let $F\in \zz[X,Y]$ be a homogeneous polynomial of degree $n\in \nn$, of content $1$ and without any repeated factors. Then there exists $N_0\in \nn$ such that for all integers $N$ divisible by $N_0$, the polynomial $\frac{F(X, NY)}{{\textrm{content}}(F(X,NY))} \in \zz[X,Y]$ has no fixed prime divisor. 
  \end{lemma}
 \begin{proof}
 Let $\alpha$ be the coefficient of the leading $X$-term of $F$. Since $\tilde{F}(X,Y):=\frac{F(X, \alpha Y)}{{\textrm{content}}(F(X,\alpha Y))}$ has leading $X$-term either equal to $X^n$ or $X^{n-1}Y$, we may as well assume without loss that $F$ itself has such leading $X$-term. 
 It then suffices to iteratively use the following fact: If $p$ is any prime number, then $\frac{F(X, pY)}{{\textrm{content}}(F(X,pY))}$ does not have $p$ as a fixed prime divisor, and has no new fixed prime divisor compared to $F(X,Y)$. Indeed, since $y\mapsto py$ is a bijection on $\zz/m\zz$ for any $m$ coprime to $p$, this transformation cannot introduce new fixed divisors; and since $\frac{F(X, pY)}{{\textrm{content}}(F(X,pY))}$ is congruent to $X^n$ or $X^{n-1}Y$ modulo $p$, it clearly does not have $p$ as a fixed divisor either.
 \end{proof}

 Use Lemma \ref{lem:pol} to replace $T$ by $S:=N_0T$; with respect to this new parameter, the ramification polynomial $F(X,Y)$ then gets replaced by $\tilde{F}(X,Y) = \frac{F(X, N_0Y)}{{\textrm{content}}(F(X,N_0Y))}$, which has no fixed prime divisor. Up to such transformation, we may therefore assume that $F(X,Y)$ itself has no fixed prime divisor.
 
With view to applications, we note here that if all branch points are finite and algebraic integers, then $F(1,0)=1$, whence the manipulations of Lemma \ref{lem:pol} are not necessary at all.
 
{\textit{Second step: Controlling the ``bad" primes}}\\
Next, in order to get full control over the discriminant of the specialization, we need to control the behaviour inside the ``exceptional set" $S_0$ for the Specialization Inertia Theorem. To that end, let $N$ be the product of all primes in $S_0$. 
Pick a finite non-branch point $t_1 = \frac{a_1}{b_1}$ with $(a_1,b_1) = 1$, such that no prime in $S_0$ divides $F(a_1,b_1)$ (possible by Step 1 and the Chinese Remainder Theorem). 
 By Krasner's lemma, for all values $t_0$ of the form $t_0 = \frac{a_1+N^m x_0}{b_1+N^m y_0}$ with $m$ sufficiently large and with $x_0,y_0\in \zz$ (i.e., $t_0$ sufficiently $N$-adically close to $t_1$), the specializations $E_{t_0}/\qq$ and $E_{t_1}/\qq$ then have the same $p$-adic behaviour at all primes $p\in S_0$. Let $\widehat{F}(X,Y) = F(a_1+N^m X, b_1+N^m Y)$. 
 Specializing  $\widehat{F}$ at integer values corresponds to specializing $E$ at values $t_0$ as above. Note that $\widehat{F}(X,Y)$ still does not have any fixed prime divisor, since, compared to $F(X,Y)$, its set of mod-$p$ values has not been restricted for any $p\notin S_0$, whereas no $p\in S_0$ is a fixed prime divisor by definition of $t_1$.
 
 {\textit{Third step: Polynomial values in prescribed residue classes}}\\
 Now we fix a prime $p\notin S_0$ and control the behaviour of $\widehat{F}(x_0,y_0)$ mod $p$. 
 Let $N'$ be the product of all primes in $S_0$ which ramify under (one, and then by Krasner automatically all) specializations $T\mapsto t_0 =\frac{a_1+N^m x_0}{b_1+N^m y_0}$  as above.
Given any non-zero residue class $k$ mod $p$, we want to enforce $\widehat{F}(x_0,y_0) \equiv k\cdot N'^{-1}$ mod $p$. This is easily achieved, after possibly enlarging the exceptional set $S_0$, by using the Hasse-Weil bound. Indeed, multiplication with $N'$ is a bijection on $\mathbb{F}_p^\times$, and the maps $x\mapsto a_1+N^m x$ and $y\mapsto b_1+N^m y$ are bijections on $\mathbb{F}_p$. 
It therefore suffices to show that, for any $k\in \mathbb{F}_p^\times$, the curve given by $F(X,Y)-k=0$ has non-trivial $\mathbb{F}_p$-points. Since $F$ is a squarefree homogeneous polynomial,\footnote{To ensure that $F$ remains squarefree modulo $p$, it suffices that no two branch points of $E/\qq(T)$ coincide modulo $p$. This could be obtained by increasing the set $S_0$; but in fact, the primes modulo which two branch points coincide are already contained in the set of exceptional primes for the Specialization Inertia Theorem.} this curve is easily seen to be absolutely irreducible \footnote{Indeed, this sort of equation is a {\textit{Thue equation}}, which is absolutely irreducible unless $F$ is a proper power, see, e.g., \cite[Thm.\ VI.9.1]{Lang}.} 
and of bounded genus (with the bound not depending on $p$ or $k$). Up to possibly enlarging the exceptional set $S_0$, we may therefore assume that this curve has non-trivial $\mathbb{F}_p$-points, 
yielding the existence of $x_0,y_0\in \zz$ with $(x_0,y_0)=1$ such that $\widehat{F}(x_0,y_0) \equiv k\cdot N'^{-1}$ mod $p$.  Recall also that our definition of $\widehat{F}$ implies that $\widehat{F}(x_0,y_0)$ is coprime to $N$.

{\textit{Fourth step: Squarefree specializations}}\\
Now we may use known criteria for squarefree specializations of $F$. Note that due to the above restrictions modulo $N^m$ and modulo $p$, we are in fact looking for integer specializations of the polynomial $F(a_1+N^m (pX+a), b_1+N^m (pY+b)) = \widehat{F}(pX+a,pY+b)$, with suitable integers $a,b$ arising from the modulo-$p$ requirements. In particular, all integer specializations of this polynomial are coprime to $N$, since $F(a_1,b_1)$ is. Also, this polynomial of course still has no fixed prime divisor, since $\widehat{F}(X,Y)$ did not and the modulo-$p$ restriction of the specialization values could only possibly introduce the fixed prime divisor $p$ (which of course is not the case, since $a$ and $b$ were such that $\widehat{F}(px+a,py+b) \equiv k\cdot N'^{-1}\ne 0$ mod $p$ for all $x,y\in \zz$).

In particular, from \cite[Theorem 3.2]{Poonen} \footnote{That theorem has as a technical assumption that $Y$ should occur in every irreducible factor of $F$, translating to the requirement that $T\mapsto 0$ should be unramified, which is however without loss up to an affine shift.}
 we obtain that, conditional on the abc-conjecture, there are infinitely many $(x_0,y_0)\in \zz^2$  such that $F(a_1+N^m (px_0+a), b_1+N^m (py_0+b))$ is squarefree. In fact, the set of such $(x_0,y_0)$ has positive density in $\zz^2$, assuming the right counting method (for our purposes, it even suffices that, upon fixing one admissible value $y_0$, the set of all $x_0$ fulfilling the above has positive density in $\zz$, which is the content of \cite[Theorem 1]{Granv}).
 
 Pick any such $x_0,y_0$ and set $t_0 = \frac{a_1+N^m (px_0+a)}{b_1+N^m (py_0+b)}$ (with automatically coprime numerator and denominator). Then the reduced discriminant $\delta$ of $E_{t_0}/\qq$ factors as $\delta = N' \cdot |F(a_1+N^m (px_0+a), b_1+N^m (py_0+b))|$. Indeed, since the value $F(a_1+N^m (px_0+a), b_1+N^m (py_0+b))$ is squarefree and not divisible by any prime in $S_0$, the Specialization Inertia Theorem shows that the ramified primes outside of $S_0$ in $E_{t_0}/\qq$ are exactly the prime divisors of $F(a_1+N^m (px_0+a), b_1+N^m (py-0+b))$. But $a$ and $b$ were chosen such that $F(a_1+N^m (px_0+a), b_1+N^m (py_0+b)) \equiv N'^{-1}\cdot k$ mod $p$. Since the density results show that the set of values $t_0$ obtained in this way is not bounded from above, we may also choose them such that $F(a_1+N^m (px_0+a), b_1+N^m (py_0+b))>0$,\footnote{Note here that $F$ can be assumed to have positive leading $X$-coefficient without loss.} showing in total $\delta \equiv k$ mod $p$.
 
 {\textit{Last step: Counting $G$-extensions}}\\
 It remains to ensure that among the specializations $E_{t_0}/\qq$ obtained as above, there are sufficiently many whose Galois group is the full group $G$. To see this, note that the admissible values $t_0$ form a set which, up to fractional linear transformation, has a subset of positive density inside some arithmetic progression (even after e.g. fixing $y_0$), i.e., positive density in the set of all integers. On the other hand, the set of integer specializations of any $\qq$-regular extension which do not yield the full Galois group is of density $0$ inside $\zz$, by Hilbert's irreducibility theorem. This shows that ``most" of the specializations $E_{t_0}/\qq$ above have group $G$. The fact that such a set (i.e., positive density inside $\zz$, up to fractional linear transformation) of specialization values $t_0$ yields asymptotically at least $B^\alpha$ distinct extensions $E_{t_0}/\qq$ with discriminant of absolute value $\le B$ with some constant $\alpha>0$ has been used several times, e.g.\ in \cite[Theorem 1.1]{Debes17} or \cite[Theorem 4.2]{Koe18}. The constant $\alpha$ here depends only on the branch point number and ramification indices of $E/\qq(T)$, and in particular not on $p$.

 Lastly, to find discriminants $\delta \equiv 0$ mod $p$ among the specializations of $E/\qq(T)$ as in ii), it suffices to use the above Specialization Inertia Theorem and the definition of $K(C)$ (see also, e.g., Remark 3.3(2) in \cite{KLN18}). The fact that the number of $G$-extensions among those is at least of the given order of growth follows just as in the last step of part i), completing the proof.

 \section{Strengthenings of Theorem \ref{thm:main}}
 \label{sec:strong}
 We note a few strengthenings of Theorem \ref{thm:main}i). Firstly, rather than considering only mod-$p$ residue classes, one may as well consider coprime residue classes modulo $p^m$ for any power $p^m$. This simply requires an additional Hensel lifting on the curve $F(X,Y)-k = 0$ in Step 3 (viewed over the field $\qq_p$). Secondly, all considerations in the above proof remain true if instead of a single prime $p\notin S_0$, one considers finitely many such primes $p_1,\dots, p_r$. This is simply due to the fact that the modulo conditions arising in the above proof are compatible by Chinese remainder theorem (note to this end that the modulo-$N$ conditions are always the same, independently of $p_i$). One therefore immediately has the following stronger version of Theorem \ref{thm:main}i).
 \begin{theorem}
 \label{thm:main_strong1}
 Assume the abc-conjecture.\\
 Let $E/\qq(T)$ be a $\qq$-regular $G$-extension, and let $A$ be an integer, not divisible by finitely many primes depending only on $E/\qq(T)$.
 Let $\mathcal{S}(E,A)$ denote the set of all $G$-extensions of $\qq$ which are specializations of $E/\qq(T)$ and which are unramified at all prime divisors of $A$.
 Let $\overline{\delta}:\mathcal{S}(E,A) \to (\zz/A\zz)^\times$ denote the mod-$A$ reduction of the reduced discriminant map. Then $\overline{\delta}$ is surjective, conditional on the abc-conjecture.
 \end{theorem}

We note furthermore that the conclusion of Theorem \ref{thm:main}i) is similar in nature to the one in \cite[Theorem 1.1]{Debes17} in the sense that both count $G$-extensions with certain local restrictions, arising via specialization of a given regular extension, although the local conditions involved are different ones. From the proofs of both theorems, it is not difficult to see that one can in fact combine those conditions to obtain the following conclusion (conditional on abc):

 \begin{theorem}
 \label{thm:main_strong2}
 Assume the abc-conjecture.\\
Let $E/\qq(T)$ be a $\qq$-regular $G$-extension, and $S_1$, $S_2$ be two finite {\textbf{disjoint}} sets of sufficiently large primes (with the implied lower bound depending on $E/\mathbb{Q}(T)$).
For each $p\in S_1$, choose a positive integer $e(p)$, a coprime residue class $k(p)$ mod $p^{e(p)}$; 
and for each $p\in S_2$, choose a conjugacy class $C(p)$ of $G$.\\
Then, among the specializations of $E/\qq(T)$, there exist infinitely many $G$-extensions $F/\qq$ (and in fact, asymptotically at least $B^\alpha$ of discriminant $\le B$, where $\alpha>0$ is a constant depending only on $E/\qq(T)$) such that
\begin{itemize}
\item[i)] the reduced discriminant of $F/\qq$ is congruent to $k(p)$ modulo $p^{e(p)}$, for all $p\in S_1$, and
\item[ii)] the Frobenius class at $p$ in $F/\qq$ equals $C(p)$, for all $p\in S_2$.
\end{itemize}
If furthermore $G$ is a perfect group, then the assertion holds even {\textbf{without}} the disjointness assumption.
\end{theorem}
\begin{proof} 
Property ii) for sufficiently large primes $p$ is shown in \cite{Debes17}. Similarly as in the proof of Theorem \ref{thm:main}i), it crucially uses the existence of simple $\mathbb{F}_p$-points on a certain curve $\overline{\mathcal{Y}}$: precisely, let $f:\mathcal{X}\to \mathbb{P}^1$ be the $\qq$-$G$-cover corresponding to $E/\qq(T)$, let $f_p = f\otimes_{\qq} \qq_p$ be the base change to $\qq_p$, and let $\varphi: G_{\qq_p}\to \langle x\rangle$ be the unramified epimorphism mapping the Frobenius to $x\in C(p)$. Following the exposition in Section \ref{sec:twists}, the curve $\overline{\mathcal{Y}}$ corresponds to (the mod-$p$ reduction of) the twisted cover $\tilde{f_p}^{\varphi}: \mathcal{Y}\to \mathbb{P}^1$. Simple $\mathbb{F}_p$-points are then lifted to $\mathbb{Q}_p$-points with $\qq$-rational image $t_0$ under $\tilde{f_p}^{\varphi}$, yielding full mod-$p$ residue classes of specializations $t_0$ such that $E_{t_0}/\qq$ fulfills property ii) at the prime $p$. The fact that properties i) and ii) for several distinct primes are compatible is then just a Chinese Remainder argument as in Theorem \ref{thm:main_strong1}.

Assume now that $G$ is perfect. Then it remains to show that, for any sufficiently large prime $p$, Conditions i) and ii) may be imposed simultaneously.
To this end, consider the curve $\mathcal{C}$ given by $F(X,Y)-k = 0$ in Step 3 of the proof of Theorem \ref{thm:main}. $\mathcal{C}$ comes with a solvable degree-$d$ cover $\mathcal{C}\to \mathbb{P}^1$, as is obvious from the homogenization $F(X,Y)-kZ^d = 0$ (where $d:=\deg(F)$). It then suffices to show that this cover has some $\mathbb{F}_p$-point at the same value $t_0 (=(x_0:y_0))$ as the cover $\overline{\mathcal{Y}}\to \mathbb{P}^1$.  Since the fiber product $\mathcal{C}\times_{\mathbb{P}^1}\overline{\mathcal{Y}}$ is of bounded genus independently of $p$, showing that it is still absolutely irreducible will suffice to complete the proof, due to Hasse-Weil. The latter holds if the corresponding two function field extensions of $\mathbb{F}_p(T)$ are linearly disjoint even over $\overline{\mathbb{F}_p}$. But one of them becomes Galois with group $G$ over $\overline{\mathbb{F}_p}(T)$, whereas the other one becomes cyclic of degree $d$. Since $G$ is perfect, these are linearly disjoint, completing the proof.
\end{proof}

\begin{remark}
Note that the assumption ``$G$ perfect" cannot be dropped for the last assertion. Indeed, let $G=C_2$.
Since for any squarefree $d\in \zz$, the reduced discriminant $\delta$ of $\qq(\sqrt{d})$ is one of $\pm d$ or $\pm 2d$, 
conditions of type i) such as $\delta \equiv k$ mod $p$, and of type ii) such as $\left(\frac{d}{p}\right) = 1$ are in general not compatible (for the same prime $p$), e.g., as soon as $k$ is a non-square mod $p$, whereas $-1$ and $2$ are squares.
\end{remark}

 We conclude this section by considering the ``usual" (i.e., non-reduced) discriminant rather than the reduced one. Statements in the direction of Theorem \ref{thm:main} have to be more intricate when considering the non-reduced discriminant, since this depends on the precise inertia groups and on the permutation action of the Galois group rather than just on the set of ramified primes. We have the following analog of Theorem \ref{thm:main_strong1}.
 \begin{theorem}
 \label{thm:main_strong3}
 Assume the abc-conjecture.\\
  Let $G\le S_n$, let $E/\qq(T)$ be a $\qq$-regular $G$-extension, and let $e:=\gcd\{{\textrm{ind}}(\sigma_i): i\in \{1,...,r\}\}$, where $\sigma_1,\dots, \sigma_r$ are the inertia group generators at the branch points of $E/\qq(T)$. Let $A$ be an integer, not divisible by finitely many primes depending only on $E/\qq(T)$.
 Let $\mathcal{S}(E,A)$ denote the set of all (degree-$n$) $G$-extensions of $\qq$ which arise via specialization of $E/\qq(T)$ and which are unramified at all prime divisors of $A$.
 Let $\overline{\Delta}:\mathcal{S}(E,A) \to (\zz/A\zz)^\times$ denote the mod-$A$ reduction of the discriminant map. Then the image of $\overline{\Delta}$ contains a non-empty union of full cosets of $((\zz/A\zz)^\times )^e$ inside $(\zz/A\zz)^\times $. 
 If furthermore, $E/\qq(T)$ possesses a trivial specialization $E_{t_0} = \qq$, then the image of $\overline{\Delta}$ contains $((\zz/A\zz)^\times )^e$.
 \end{theorem}
 \begin{proof}
 First recall that the multiplicity of any tamely ramified prime $p$ in a $G$-extension equals the index $ind(\sigma)$ of its inertia group generator $\sigma$. Indeed,  the ramification indices $e_{\mathfrak{p}}$ at prime ideals $\mathfrak{p}$ extending $p$ in that extension are exactly the orbit lengths of the inertia subgroup at $p$ (e.g., \cite[Thm.\ I.9.1]{MM}), and the sum $\sum_{\mathfrak{p}|p} (e_{\mathfrak{p}}-1)$ equals the exponent of $p$ in the discriminant by Dedekind's Different Theorem, e.g., \cite[Thm.\ III.2.6]{Neukirch}. On the other hand, due to tame ramification, $\sum_{\mathfrak{p}|p} (e_{\mathfrak{p}}-1)$ is also exactly the index of an inertia group generator at $p$, since the index of an $e_{\mathfrak{p}}$-cycle is $e_{\mathfrak{p}}-1$.
 
 Therefore, after a suitable modification of the definition of $N'$ in Step 3,\footnote{In particular, since the non-reduced discriminant comes with a sign in general, one should include that sign in the definition of $N'$ and then add an archimedean condition, i.e., restrict specialization values to a certain open real interval in order to preserve the sign.} the only necessary adaptations in the proof of Theorem \ref{thm:main} are the following: instead of considering the curve $F(X,Y)-k = 0$ over $\mathbb{F}_p$, with the squarefree homogeneous polynomial $F$, one now should consider the curve $C: (\prod_{i=1}^s f_i(X,Y)^{\nu_i}) -k = 0$, where $F = \prod_{i=1}^s f_i$ is a factorization into irreducible homogeneous polynomials and $\nu_i = \textrm{ind}(\sigma_i)$ is the index of an inertia group generator corresponding to $f_i$.
 Obviously, this curve cannot have an $\mathbb{F}_p$-rational point if $k\in \mathbb{F}_p^\times$ is not an $e$-th power. Conversely, as soon as $k = \kappa^e$ with $\kappa\in \mathbb{F}_p^\times$, this curve has a component $\tilde{C}: (\prod_{i=1}^s f_i(X,Y)^{\nu_i/e}) -\kappa = 0$. Since $\gcd\{\nu_1/e,\dots, \nu_s/e\} =1$ by definition, the polynomial $\prod_{i=1}^s f_i(X,Y)^{\nu_i/e}$ is not a proper power, and therefore $\tilde{C}$ is absolutely irreducible. For the proof of the first assertion, the rest of the argument remains unchanged. For the second one, note that the existence of a trivial specialization ensures that in the second step of the proof of Theorem \ref{thm:main}, one may restrict (via Krasner's lemma) to specializations in which all ``bad" primes are unramified, yielding $N'=1$ in the third step. This means that the coset of $((\zz/A\zz)^\times )^e$ occurring in the first assertion can be chosen to be the coset of the identity.
 \end{proof}
 
 In the light of Theorem \ref{thm:main_strong3}, the following analog of Conjecture \ref{conj:localmalle} for non-reduced discriminants suggests itself:
 \begin{conjecture}
 \label{conj:localmalle2}
 Let $G\le S_n$ be a transitive permutation group, and let $e_0 = \gcd\{{\textrm{ind}}(g): g\in G\setminus\{1\}\}$.
 Given any prime $p$, not in some finite set $S_0$ depending on $G$, and an integer $k \not\equiv 0$ mod $p$,
there exists a $G$-extension of $\qq$ whose discriminant $\Delta$ fulfills $\Delta \equiv k^{e_0}$ mod $p$. \end{conjecture}

Theorem \ref{thm:main_strong3} then answers Conjecture \ref{conj:localmalle2} in the positive (conditional on abc), as long as $G$ occurs as the Galois group of a regular Galois extension $E/\qq(T)$ satisfying the following two conditions:
\begin{itemize}
\item[i)] The gcd of indices of inertia group generators in $E/\qq(T)$ equals $e_0$, i.e., is smallest possible.
\item[ii)] $E/\qq(T)$ possesses a trivial specialization $E_{t_0} = \qq$.
\end{itemize}
Condition i) is fulfilled if $E/\qq(T)$ contains every non-identity conjugacy class as a class of inertia group generators; however, this very strong condition is usually far from necessary. E.g., for degree-$n$ extensions with group $A_n$, it is sufficient if $E/\qq(T)$ has a branch point with inertia group generated by a $3$-cycle; such extensions are known to exist due to a famous construction by Mestre (see e.g. \cite[Chapter IV.5.4]{MM}). Theorem \ref{thm:main_strong3}, together with the obvious fact that every $A_n$-discriminant is a square, then yields that Conjecture \ref{conj:localmalle2} is true for $G=A_n$. The same holds for $G=S_n$, with a similar, but easier construction (see, e.g., Lemma 5.1 in \cite{Koe18}).
 
 \section{Unconditional results}
 \label{sec:uncond}
 It may be desirable to have a variant of Theorem \ref{thm:main} which holds without assuming the abc-conjecture, at the cost of certain extra conditions on the given function field extensions. This is obtained by means of the following lemma, of which the special case $G=C_2$ is once again contained in \cite{KP18} (Theorem 3). 
 \begin{lemma}
 \label{lem:degle3}
 Let $E/\qq(T)$ be a $\qq$-regular Galois extension all of whose branch points have degree at most $6$ over $\qq$. Then the assertions of Theorems \ref{thm:main}i), \ref{thm:main_strong1}, \ref{thm:main_strong2} and \ref{thm:main_strong3} hold, not conditional on the abc-conjecture. 
 \end{lemma}
 \begin{proof}
 The following is known unconditionally (see \cite{Greaves} with some technical extra condition, and \cite{YX} in full): If $F(X,Y)\in \zz[X,Y]$ is a binary form without repeated factors and with no irreducible factor of degree $>6$, and if $A,B,M$ are positive integers such that no integer square $>1$ divides all values of $F(MX+A, MY+B)$, then the set of $(x_0,y_0)\in \zz^2$ such that $F(Mx_0+A, My_0+B)$ is squarefree is infinite; more precisely, the cardinality of its intersection with $\{-n,\dots, n\}^2$ is asymptotically bounded from below by $cn^2$, for some positive constant $c$.\footnote{In particular, this density result implies also that the quotient $\frac{Mx_0+A}{My_0+B}$ may be chosen of arbitrarily large absolute value which is technically needed in Step 4 to control the sign of the values of $F$.}\\
 The abc-conjecture was used in the proof of Theorem \ref{thm:main}i) only to ensure that a certain polynomial $F(a_1+N^m (pX+a), b_1+N^m (pY+b))$
 (where the homogeneous polynomial $F$ is the product of minimal polynomials of the branch points), known by construction not to possess any fixed square divisors, obtains squarefree values for a subset of the arguments which is of larger order of growth then the ``exceptional Hilbert set" (i.e., the set of specialization values which change the Galois group). We may therefore replace the abc-conjecture by the above result, with $M = N^mp$, $A=a_1+N^ma$ and $B=b_1+N^mb$, to obtain the assertion. 
 \end{proof} 
  
 Many groups have realizations fulfilling the assumptions of Lemma \ref{lem:degle3}, making the assertion of Theorem \ref{thm:main}i) true at least for {\textit{some}} $\qq$-regular $G$-extension.
 
 \begin{theorem}
 \label{thm:examples}
Let $G$ be a direct product of finitely many factors, all of which are one of 
$S_n$, $A_n$ ($n\ge 3$); $PSL_2(p)$ (where $p\in \mathbb{P}$ is such that $\left(\frac{2}{p}\right)=-1$ or $\left(\frac{3}{p}\right)=-1$);   or any sporadic simple group other than $M_{23}$. 
 Then there exists a $\qq$-regular $G$-extension $E/\qq(T)$ fulfilling the assertion of Theorem \ref{thm:main}, not conditional on the abc-conjecture.
In particular, Conjecture \ref{conj:localmalle} holds for $G$.
\end{theorem}
\begin{proof}
Firstly, we may obviously reduce to the case of a single factor, simply because the regular inverse Galois problem is well-behaved under taking direct products.
Then, for all the groups given above, the rigidity method and certain extensions provide regular Galois extensions over $\qq$ with either three or four branch points; and from the ramification data it is always obvious that there exists at least one rational branch point (yielding conclusion ii) of Theorem \ref{thm:main}). 
Indeed, apart from the groups $M_{11}$ and $M_{24}$, the realizations in \cite{MM} yield three-branch-point realizations $E/\qq(T)$ for either $G$ itself or some group $\Gamma$ containing $G$ of index $2$ (see in particular Theorem I.7.9 and Section II.9 of \cite{MM} for the linear and sporadic groups respectively). In the latter case, one deduces suitable $G$-realizations via standard arguments. We illustrate this with the case $G=A_n$.

It is well-known that the polynomial $X^n - T(nX-(n-1))$ defines a $\qq$-regular $S_n$-extension $E/\qq(T)$ with three rational branch points ($0$, $\infty$ and $1$) and with inertia groups generated by an $n$-cycle, an ($n-1$)-cycle and a transposition respectively. Since exactly two of the inertia group generators lie outside of $A_n$ (namely the transposition and one of the other two, depending on whether $n$ is even or odd), the quadratic extension $E^{A_n}/\qq(T)$ is ramified at exactly two rational points, and is therefore a rational function field, say $\qq(S)$. Then, $E/\qq(S)$ is an $A_n$-extension ramified only over the points extending $T\mapsto 0$ and $T\mapsto \infty$ (which is a total of three points, one of them rational, due to the ramification conditions in $\qq(S)/\qq(T)$).

Finally, for $M_{11}$ and $M_{24}$, $\qq$-regular realizations with four branch points (including a rational one) are given in \cite{Koe_m11} and \cite[III.7.5]{MM} respectively.
 \end{proof}
 
\section{On the set of exceptional primes}
\label{sec:exc}
Next, we consider more closely the set $S_0$ of exceptional primes in Conjecture \ref{conj:localmalle}. We do not try to be exhaustive in this section, but rather to include examples which are interesting from a group-theoretic or arithmetic-geometric point of view.
 In general, $S_0$ cannot be expected to be empty, since the trivial local obstructions $0$ mod $p$ for $p\notin \mathcal{P}(G)$ may yield modulo-$q$ conditions for all prime divisors of the discriminant of a $G$-extension, where $q$ is a prime divisor of $|G|$. Since those obstructions only concern prime divisors of $|G|$, it may be expected that Conjecture \ref{conj:localmalle} can be strengthened by adding the assumption $S_0\subset \{p\in \mathbb{P}: |G| \equiv 0 \ {\textrm{ mod }}\  p\}$.  The following result confirms this expectation for Galois groups of {\textit{Belyi maps}} $f:X\to \mathbb{P}^1$ defined over $\qq$ (upon possibly adding the prime $7$, which is however a divisor of $|G|$ anyway for many interesting groups). Recall that $f$ is called a Belyi map if it is ramified exactly at $0$, $1$ and $\infty$.
 \begin{theorem}
   \label{thm:explicit}
 Let $G$ be a finite group. Assume that  there exists a $\qq$-$G$ cover $f:X\to \mathbb{P}^1$ which is a Belyi map 
 and a ``good model"\footnote{See \cite[Def.\ 2.2]{Beckmann}. Instead of going into details of the definition of a good model, we note here that this is automatically fulfilled if $G$ has trivial center, and also (after a certain twist) if $Z(G)$ is of exponent $2$, see \cite[Prop. 2.4]{Beckmann}.}  over $\mathbb{Q}$.
 Then Conjecture \ref{conj:localmalle} holds (unconditionally) for $G$ with the exceptional set $S_0$ chosen to be the set of prime divisors of $7|G|$.
   (and one may restrict to $G$-extensions arising as specializations of $f$). 
 \end{theorem}
 \begin{proof}
 Let the $\qq$-regular Galois extension $E/\qq(T)$ be the function field extension of $f$.
 Once again, we need to make suitable modifications to the proof of Theorem \ref{thm:main}. From the proof, the only primes $p$ that need to be exempted are those which are exceptional for the Specialization Inertia Theorem, and the ones such that a certain curve $F(X,Y)-k = 0$ does not have any non-trivial $\mathbb{F}_p$-points.
 
 First, let us describe explicitly the former set of primes. Since $E/\qq(T)$ is ramified exactly at $T\mapsto 0$, $T\mapsto 1$ and $T\mapsto \infty$, the polynomial $F$ is actually $XY(X-Y)$. We may then make the finite set $S_0$ of exceptional primes for the Specialization Inertia Theorem explicit: all exceptional primes $p$ are either divisors of $|G|$, or are primes of ``vertical ramification", 
 or ramified in the residue field of some branch point, 
 or such that two or more branch points coincide modulo $p$ (see, e.g., \cite[Section 2.2.2]{Legrand}). 
With our assumptions, the last two cases obviously do not occur, whereas the second one does not occur (outside the set of prime divisors of $|G|$) due to the definition of a good model. 
This shows that the only exceptional primes that need to be excluded due to the Specialization Inertia Theorem are the prime divisors of $|G|$.

 Next, we deal with the second class of primes. These are the primes $p$ not dividing $|G|$ for which the curve $\mathcal{C}: F(X,Y) - k=0$ is either not absolutely irreducible over $\mathbb{F}_p$ or does not possess non-trivial $\mathbb{F}_p$-points. 
 Note that the polynomial $XY(X-Y)$ has a fixed divisor $2$, whence, for the sake of preserving the general outline of the proof in Section \ref{sec:proof}, we use a linear transformation to turn the branch points into $T\mapsto 0$, $T\mapsto 2$ and $T\mapsto \infty$, which gets rid of fixed divisors at the cost of making $2$ a bad prime.
 Note also that due to our assumptions, $2$ must divide $|G|$, since a $\qq$-$G$-cover for an odd order group $G$ cannot have a rational (or even real) branch point, cf.\ \cite[Corollary 1.3]{DF90}.
So consider the polynomial $F=XY(X-2Y)$, which for any $p\ne 2$ is a separable cubic polynomial modulo $p$. 
Note also that due to our assumptions, $2$ divides $|G|$, since a $\qq$-$G$-cover for an odd order group $G$ cannot have a rational (or even real) branch point, cf.\ \cite[Corollary 1.3]{DF90}.
For $p> 3$, the curve $\mathcal{C}$ is then of genus $1$, 
  since its definition yields a degree-$3$ cover $\mathcal{C}\to \mathbb{P}^1$, ramified (totally) over exactly three points (namely the points corresponding to $Z=0$ in the homogenized equation $F(X,Y) - kZ^3=0$. Since $\mathcal{C}$ is absolutely irreducible for $p\not{|} |G|$, Hasse's bound then yields that $\mathcal{C}$ has at least $p+1 - 2\sqrt{p}$ $\mathbb{F}_p$-points, of which only three are trivial. For $p\ge 11$, this yields non-trivial points, and in fact for $p\in \{3,5\}$, non-trivial $\mathbb{F}_p$-points $(x:y:z)$ exist as well (e.g. with $x=y \ne 0$,
   since $z\mapsto z^3$ is then bijective on $\mathbb{F}_p^\times$). 
   This shows altogether that all exceptional primes divide $7|G|$, whence the assertion for $k\ne 0$ mod $p$ follows from Lemma \ref{lem:degle3}. Finally, the case $k\equiv 0$ mod $p$ follows from the existence of a rational branch point, as in Theorem \ref{thm:main}ii).
 \end{proof}

 \begin{example}
 \label{ex:monster}
 A purely group-theoretical way to obtain Belyi maps over $\qq$ is via the rigidity method (see, e.g., Chapter I of \cite{MM}). A famous example concerns the sporadic Monster group $M$. This group is known to fulfill all the assumptions of Theorem \ref{thm:explicit} due to Thompson (\cite{Thompson}), and has order divisible by $7$. Therefore Conjecture \ref{conj:localmalle} holds (unconditionally) for $M$ outside of the prime divisors of $|M|$. 
 \end{example}

\begin{example}
\label{ex:an}
The assumption of rational branch points in Theorem \ref{thm:explicit}, implicit in the definition of a Belyi map over $\qq$, can be relaxed somewhat. For example, let $G=A_n$. The $A_n$-extension occuring in the proof of Theorem \ref{thm:examples} can be verified to be ramified (up to fractional linear transformation) exactly at $0$ and either $\pm\sqrt{(-1)^{(n-1)/2}n}$ or $\pm\sqrt{(-1)^{(n-2)/2}(n-1)}$, depending on whether $n$ is odd or even (see also the appendix of \cite{MM} for the explicit corresponding $A_n$-polynomial). The set of exceptional primes for the Specialization Inertia Theorem is then still only the set of prime divisors of $|G|$. The rest of the proof of Theorem \ref{thm:explicit} goes through without changes. In particular, for all $n\ge 7$, Conjecture \ref{conj:localmalle} holds for $A_n$, outside the set of primes $\le n$.\footnote{A database search of $A_n$-extensions for small $n$ quickly confirms that in fact the restriction $n\ge 7$ is not necessary.}
\end{example}

 In certain scenarios, we can do even better regarding the restriction of the set of exceptional primes. As an example we deal with the case of {\textit{polynomial maps}} $\mathbb{P}^1\to \mathbb{P}^1$, which is of importance e.g.\ in arithmetic dynamics.
 
 \begin{theorem}
 \label{thm:pol}
Assume that the $\qq$-regular $G$-extension $E/\qq(T)$ is the Galois closure of a rational function field extension $\qq(x)/\qq(T)$ corresponding to a degree-$n$ {\textit{polynomial}} map $f:\mathbb{P}^1\to \mathbb{P}^1$, $x\mapsto q(x)$ with $q \in \qq[X]$. Assume that either the abc-conjecture holds, or $q$ has at most $7$ distinct roots in $\overline{\qq}$. Then Conjecture \ref{conj:localmalle} holds for $G$ with exceptional set the prime divisors of $n$.
 \end{theorem}
 \begin{proof}
  By assumption, $E/\qq(T)$ is the splitting field of a degree-$n$ polynomial $q(X)-T\in \qq(T)[X]$, where $q(X)$ is monic and integral without loss, via suitable scalings in $X$ and $T$. 
 Let $F(X,Y)\in \zz[X,Y]$ be the homogeneous polynomial corresponding to the branch point locus of $E/\qq(T)$, as before. 

Now fix a prime $p$ not dividing $n$. To get rid of fixed prime divisors of $F(X,Y)$, we use Lemma \ref{lem:pol}. Let $N_0$ be as in that lemma, and set $S:=T\cdot (pN_0)^n$. Then $E/\qq(S)$ is the splitting field of $(pN_0)^n\cdot q(\frac{X}{pN_0}) - S$, and by Lemma \ref{lem:pol}, the ramification polynomial for this extension, which we denote again by $F(X,Y)$, has no fixed divisors. Also note that the polynomial $\tilde{q}(X):=(pN_0)^n\cdot q(\frac{X}{pN_0})$ is still monic and integral, and with all except the leading coefficient divisible by $p$.
 We fix a residue class $k$ mod $p$, and as before want to find specializations of $E/\qq(S)$ with reduced discriminant congruent to $k$ mod $p$. If $k\equiv 0$ mod $p$, simply specialize $S\mapsto s_0$ of $p$-adic valuation $1$, which yields a $p$-Eisenstein polynomial. If $k\ne 0$ mod $p$, every specialization $S\mapsto s_0$ of $p$-adic valuation $0$ leads to a separable polynomial $f(X)-s_0$ modulo $p$, i.e., an extension unramified at $p$ (we have used here that $p$ does not divide $n$). 
Since $\tilde{q}(X) \equiv X^n$ mod $p$, all finite branch points of $E/\qq(S)$ coincide with $S\mapsto 0$ modulo $p$, so the modulo-$p$ reduction of $F(X,Y)$ equals $X^dY$ (for some $d\in \nn$). We then need to find non-trivial $\mathbb{F}_p$-points on the curve given by $X^dY-k=0$.\footnote{Note that the inseparability of $F$ mod $p$ means that $p$ is a bad prime! This is however not a problem here, since we have controlled the behaviour at $p$ via an explicit polynomial.} This, however, is a rational genus-$0$ curve, and therefore automatically has non-trivial $\mathbb{F}_p$-points! We therefore obtain a mod-$p$ congruence condition for specialization values $s_0=(x_0:y_0)$, and we can add another congruence condition (as in Step 2 in Section \ref{sec:proof}) in order to control the ramification behaviour at bad primes {\textit{other than}} $p$.
 Using once again the abc-conjecture, we find many squarefree values of $F$ under the above congruence conditions; and unconditionally for $q$ with at most $7$ distinct roots (as in Lemma \ref{lem:degle3}), since by the Riemann-Hurwitz formula $E/\qq(S)$ can then have at most $6$ branch points other than the rational branch points at $S\mapsto \infty$ and (possibly) $S\mapsto 0$. This concludes the proof.
 \end{proof}
 
 \begin{remark}
 \label{rk:pol0}
 The structure of the Galois group of a polynomial map as in Theorem \ref{thm:pol} is somewhat restricted. In particular, in the case of {\textit{indecomposable}} polynomials, a full classification of all possibilities was obtained by M\"uller in \cite{Mue}. Nevertheless, in the decomposable case, interesting classes of groups, such as iterated wreath products of symmetric groups, occur.
 \end{remark}
 
 \begin{remark}
 \label{rk:pol}
 It is an easy exercise to verify that the conclusion of Theorem \ref{thm:pol} remains true when the polynomial $q(X)$ is replaced by a rational function of the form $\frac{q(X)}{X}$, the exceptional set $\{p\in \mathbb{P} : p|n\}$ is replaced by $\{p\in \mathbb{P} : p|n-1\}$, and the maximal number of distinct roots of $q$ in the unconditional case is replaced by $6$.\footnote{The only relevant change in the proof concerns generation of extensions ramified at $p$; here the Eisenstein argument should be replaced with a generalization involving Newton polygons.}
 \end{remark}
 
 We end with an example where we can completely get rid of exceptional primes.
 \begin{example}
 \label{ex:sn}
 Let $n\ge 3$, and let $E/\qq(T)$ be the splitting field of the trinomial $X^n+TX+T$. This is $\qq$-regular with Galois group $S_n$, see the proof of Theorem \ref{thm:examples}.
 Furthermore $E/\qq(T)$ is the splitting field of both a polynomial of the form $q(X)-T$ and one of the form $\tilde{q}(X)-TX$, with $q,\tilde{q}\in \qq[X]$, via linear transformations (this can also be seen from the ramification data, since the ramification type is of genus $0$, containing both an $n$-cycle and an $n-1$-cycle).
 It then follows from Theorem \ref{thm:pol} and Remark \ref{rk:pol} that Conjecture \ref{conj:localmalle} holds for $S_n$ with no exceptional primes, and one may even restrict to trinomial extensions. Note also that this result is unconditional, since $E/\qq(T)$ has only three branch points.
\end{example}
 
 {\textbf{Acknowledgement}}: I would like to thank David Krumm and Fran\c{c}ois Legrand, as well as the anonymous referee for helpful comments and suggestions.


\begin{thebibliography}{9}
 \bibitem{Beckmann} S.\ Beckmann, \textit{On extensions of number fields obtained by specializing branched coverings}. J.\ Reine Angew.\ Math. 419 (1991), 27--53.
 \bibitem{Bhargava} M.\ Bhargava, \textit{The density of discriminants of quartic rings and fields}. Ann.\ of Math.\ 162 (2005), 1031--1063. 
 \bibitem{Bhargava2} M.\ Bhargava, \textit{The density of discriminants of quintic rings and fields}. Ann.\ of Math.\ 172 (2010), 1559--1591.
 \bibitem{Debes17} P.\ D\`ebes, \textit{On the Malle conjecture and the self-twisted cover}. Israel J.\ Math.\ 218 (2017), no.\ 1, 101--131. 
 \bibitem{DF90} P.\ D\`ebes, M.D.\ Fried, \textit{Rigidity and real residue class fields}. Acta Arithmetica 56 (1990), 291--323.
 \bibitem{EV} J.\ Ellenberg, A.\ Venkatesh, \textit{Counting extensions of function fields with bounded discriminant and specified Galois group}. In: Geometric methods in algebra and number theory, 151--168, Progr. Math., 235, Birkh\"auser Boston, Boston, MA, 2005.
 \bibitem{Granv} A.\ Granville, \textit{ABC allows us to count squarefrees}. Internat.\ Math.\ Res.\ Notices 1998, no.\ 19, 991--1009.
 \bibitem{Greaves} G.\ Greaves, \textit{Power-free values of binary forms}. Quart.\ J.\ Math.\ Oxford Ser.\ (2) 43 (1992), no. 169, 45--65.
 \bibitem{KlM} J.\ Kl\"uners, G.\ Malle, \textit{Counting nilpotent Galois extensions}. J.\ Reine Angew.\ Math.\ 572 (2004), 1--26.
 \bibitem{Koe_m11} J.\ K\"onig, \textit{On rational functions with monodromy group $M_{11}$}. J.\ Symb.\ Comp.\ 79 (2) (2017), 372--383.
 \bibitem{Koe18} J.\ K\"onig, \textit{On number fields with $k$-free discriminant}. Preprint (2018). \texttt{https://arxiv.org/abs/1809.01861}
 \bibitem{KL18} J.\ K\"onig, F.\ Legrand, \textit{Non-parametric sets of regular realizations over number fields}. J.\ Algebra 497 (2018), 302--336.
 \bibitem{KL19} J.\ K\"onig, F.\ Legrand, \textit{Density zero results for sets of specializations of Galois covers}. In preparation.
 \bibitem{KLN18} J.\ K\"onig, F.\ Legrand, D.\ Neftin, \textit{On the local behavior of specializations of function field extensions}. To appear in Int.\ Math.\ Res.\ Not.\ IMRN. Preprint at \texttt{https://arxiv.org/abs/1709.03094}
 \bibitem{Krumm} D.\ Krumm, \textit{Squarefree parts of polynomial values}. J. Th\'eor.\ Nombres Bordeaux 28 (2016), no. 3, 699--724.
 \bibitem{KP18} D.\ Krumm, P.\ Pollack, \textit{Twists of hyperelliptic curves by integers in progressions modulo $p$}. Preprint (2018). \texttt{https://arxiv.org/abs/1807.00972}
 \bibitem{Lang} S.\ Lang, \textit{Algebra - revised third edition}. Graduate Texts in Mathematics 211, Springer, New York, 2002.
 \bibitem{Legrand} F.\ Legrand, \textit{Specialization results and ramification conditions}. Israel J.\ Math 214 (2) (2016), 621--650.
 \bibitem{Malle02} G.\ Malle, \textit{On the distribution of Galois groups}. J.\ Number Theory 92 (2002), no. 2, 315--329.
 \bibitem{Malle04} G.\ Malle, \textit{On the distribution of Galois groups II}.  Experiment.\ Math.\ 13 (2) (2004), 129--136.
 \bibitem{MM} G.\ Malle, B.H.\ Matzat, \textit{Inverse Galois Theory}. Second edition. Springer Monographs in Mathematics. Springer, Berlin, 2018. xvii+532 pp.
 \bibitem{Mue} P.\ M\"uller, \textit{Primitive monodromy groups of polynomials}. In: Recent developments in the inverse Galois problem, Amer.\ Math. Soc., Providence (1995), 385--401.
 \bibitem{Neukirch} J.\ Neukirch, \textit{Algebraic number theory}. Springer, Berlin-Heidelberg-New York, 1999.
 \bibitem{Poonen} B.\ Poonen, \textit{Squarefree values of multivariable polynomials}. Duke Math.\ J.\ 118 (2003), no. 2, 353--373.
 \bibitem{Thompson} J.G.\ Thompson,  \textit{Some finite groups which appear as ${\textrm{Gal}}(L/K)$, where $K \subseteq \mathbb{Q}(\mu_n)$}. J.\ Algebra, 89 (2) (1984), 437--499.
 \bibitem{Wright} D.J.\ Wright, \textit{Distribution of discriminants of abelian extensions}. Proc.\ London Math.\ Soc.\ 58 (1) (1989). 17--50.
 \bibitem{YX} S.Y.\ Xiao, \textit{Power-Free Values of Binary Forms and the Global Determinant Method}. Int.\ Math.\ Res.\ Not.\ IMRN Volume 2017, Issue 16 (2017), 5078--5135.
  \end{thebibliography}
 \end{document}